\documentclass[11pt]{article}
\usepackage{amsmath, amssymb, amsfonts, amsthm, latexsym} 
\usepackage{mathrsfs}
\usepackage{graphicx}
\usepackage{authblk}
\usepackage[usenames]{color}
\newtheorem{main}{Theorem}
\newtheorem{theorem}{Theorem}[section]
\newtheorem{lemma}[theorem]{Lemma}

\newtheorem{conjecture}[theorem]{Conjecture}

\newtheorem{corollary}[theorem]{Corollary}
\newtheorem{remark}[theorem]{Remark}
\newtheorem{observation}[theorem]{Observation}

\newtheorem*{notat*}{Notation}

\title{Cycle lengths in randomly perturbed graphs}

\author{Elad Aigner-Horev \thanks{Department of Computer Science, Ariel University, Ariel 40700, Israel. Email: {\tt horev@ariel.ac.il}.}
\quad Dan Hefetz \thanks{Department of Computer Science, Ariel University, Ariel 40700, Israel. Email: {\tt danhe@ariel.ac.il}. Research supported by ISF grant 822/18.}
\quad Michael Krivelevich \thanks{School of Mathematical Sciences, Tel Aviv University, Tel Aviv 6997801, Israel. Email: {\tt krivelev@tauex.tau.ac.il}. Research supported in part by USA--Israel BSF grant 2018267.}}

\begin{document}
\maketitle

\begin{abstract}
Let $G$ be an $n$-vertex graph, where $\delta(G) \geq \delta n$ for some $\delta := \delta(n)$. A result of Bohman, Frieze and Martin from 2003 asserts that if $\alpha(G) = O \left(\delta^2 n \right)$, then perturbing $G$ via the addition of $\omega \left(\frac{\log(1/\delta)}{\delta^3} \right)$ random edges, asymptotically almost surely (a.a.s. hereafter) results in a Hamiltonian graph. This bound on the size of the random perturbation is only tight when $\delta$ is independent of $n$ and deteriorates as to become uninformative when $\delta = \Omega \left(n^{-1/3} \right)$.  

We prove several improvements and extensions of the aforementioned result. First, keeping the bound on $\alpha(G)$ as above and allowing for $\delta = \Omega(n^{-1/3})$, we determine the correct order of magnitude of the number of random edges whose addition to $G$ a.a.s. results in a pancyclic graph. Our second result ventures into significantly sparser graphs $G$; it delivers an almost tight bound on the size of the random perturbation required to ensure pancyclicity a.a.s., assuming $\delta(G) = \Omega \left((\alpha(G) \log n)^2 \right)$ and $\alpha(G) \delta(G) = O(n)$. 

Assuming the correctness of Chv\'atal's toughness conjecture, allows for the mitigation of the condition $\alpha(G) = O \left(\delta^2 n \right)$ imposed above, by requiring $\alpha(G) = O(\delta(G))$ instead; our third result determines, for a wide range of values of $\delta(G)$, the correct order of magnitude of the size of the random perturbation required to ensure the a.a.s. pancyclicity of $G$.

For the emergence of nearly spanning cycles, our fourth result determines, under milder conditions, the correct order of magnitude of the size of the random perturbation required to ensure that a.a.s. $G$ contains such a cycle.
\end{abstract}

\section{Introduction}  

A \emph{Hamilton cycle} in a graph $G$ is a cycle passing through all vertices of $G$. A graph is said to be \emph{Hamiltonian} if it contains a Hamilton cycle. Hamiltonicity is one of the most central notions in graph theory, and has been studied extensively in the last 70 years. Considerable effort has been devoted to obtaining  sufficient conditions for the existence of a Hamilton cycle (an effective necessary and sufficient condition should not be expected however, as deciding whether a given graph contains a Hamilton cycle is known to be NP-complete). One of the earliest results in this direction is the celebrated theorem of Dirac~\cite{Dirac}, asserting
that every $n$-vertex graph $H$ (on at least three vertices) with minimum degree $\delta(H) \geq n/2$ is Hamiltonian. Since then, many other sufficient conditions that deal with dense graphs were obtained (see, e.g., the comprehensive references~\cite{Gould, KO}). When looking into the Hamiltonicity of sparser  graphs, it is natural to consider random graphs with an appropriate edge probability. Erd\H{o}s and R\'enyi~\cite{ER2} raised the problem of determining the threshold probability of Hamiltonicity in the binomial random graph model $\mathbb{G}(n,p)$. Following a series of efforts by various researchers, including Korshunov~\cite{Korshunov} and P\'osa~\cite{Posa}, the problem was finally solved by Koml\'os and Szemer\'edi~\cite{KSz} and independently by Bollob\'as~\cite{Bollobas}, who proved that if $p := p(n) \geq (\log n + \log \log n + \omega(1))/n$, where the $\omega(1)$ term tends to infinity with $n$ arbitrarily slowly, then $G \sim \mathbb{G}(n, p)$ is asymptotically almost surely (a.a.s. for brevity, hereafter) Hamiltonian. This is best possible since for $p \leq (\log n + \log \log n - \omega(1))/n$ a.a.s. there are vertices of degree at most one in $G \sim \mathbb{G}(n, p)$.

An $n$-vertex graph is said to be \emph{pancyclic} if it contains a cycle of length $\ell$ for every $3 \leq \ell \leq n$. Every pancyclic graph is Hamiltonian whereas the converse implication does not hold in general. A meta-conjecture of Bondy~\cite{BondyConj} asserts that almost any non-trivial sufficient condition for Hamiltonicity ensures pancyclicity as well (a simple family of exceptional graphs may exist). This turns out to be true for the condition appearing in Dirac's Theorem --- it was proved by Bondy himself~\cite{Bondy} that every $n$-vertex graph $H$ with minimum degree $\delta(H) \geq n/2$ is either $K_{n/2, n/2}$ or pancyclic. Bondy's meta-conjecture holds true for random graphs as well --- Cooper and Frieze~\cite{CP} proved that $G \sim \mathbb{G}(n,p)$ is a.a.s. pancyclic whenever $p := p(n) \geq (\log n + \log \log n + \omega(1))/n$.

The aforementioned theorem of Dirac is optimal in terms of the minimum degree, that is, for every $0 < \delta := \delta(n) < 1/2$, there are non-Hamiltonian graphs with minimum degree at most $\delta n$. Nevertheless, Bohman, Frieze, and Martin~\cite{BFM} discovered that once {\sl slightly randomly perturbed}, graphs with linear minimum degree become Hamiltonian asymptotically almost surely. Formally, they proved that for every constant $\delta > 0$ there exists a constant $C := C(\delta)$ such that $H \cup R$ is a.a.s. Hamiltonian, whenever $H$ is an $n$-vertex graph with minimum degree at least $\delta n$ and $R \sim \mathbb{G}(n, C/n)$; undershooting the threshold for Hamiltonicity in $\mathbb{G}(n,p)$ by a logarithmic factor. Such a result can be seen as \emph{bridging} between results regarding the  Hamiltonicity of dense graphs and the emergence of such cycles in random graphs. Numerous results related to Hamiltonicity of randomly perturbed graphs (and hypergraphs) have since appeared; see, e.g.,~\cite{AF, AH, BHKM, DRRS, HZ, KKS, MM}.    

Bohman, Frieze, and Martin noted in~\cite{BFM} that their aforementioned result is tight (up to a constant factor) as there are $n$-vertex graphs with minimum degree $\Omega(n)$ which cannot be made Hamiltonian (even deterministically) by the addition of $c n$ edges for some sufficiently small constant $c > 0$. Indeed, consider for example the complete bipartite graph $G := K_{n/3, 2n/3}$. It is evident that $e(H) \geq n/3$ for any graph $H$ with vertex-set $V(G)$ for which $G \cup H$ is Hamiltonian. It is suggested in~\cite{BFM} that the reason for $K_{n/3, 2n/3}$ requiring so many additional edges to become Hamiltonian is that it admits a large independent set. Subsequently, they have proved the following result.

\begin{theorem} [Theorem 2 in~\cite{BFM} -- abridged] \label{th::BFMalpha}      
Let $G$ be an $n$-vertex graph with minimum degree at least $\delta n$ for some $\delta = \delta(n)$. Let $R \sim {\mathbb G}(n,p)$, where $p := p(n) = \omega \left( \frac{\log(1/\delta)}{\delta^3 n^2} \right)$. If $\alpha(G) < \delta^2 n/2$, then $G \cup R$ is a.a.s. Hamiltonian.
\end{theorem}

The bound on $p$ stated in Theorem~\ref{th::BFMalpha} is sharp whenever $\delta$ is constant. However, it is not sharp for $\delta = o(1)$. In fact, for $\delta = O(n^{- 1/3})$, this bound is larger than the threshold for Hamiltonicity in ${\mathbb G}(n,p)$, making it useless in this regime. Our first result determines the correct value of $p$. With some abuse of notation, throughout the paper we denote a union of the form $G \cup R$, where $G$ is a given $n$-vertex graph and $R \sim {\mathbb G}(n,p)$, by $G \cup {\mathbb G}(n,p)$.

\begin{main} \label{th::smallAlpha} 
There exist positive constants $c_1, c_2, c_3$ and $c_4$ such that the following hold.
\begin{description}
\item [(a)] Let $G$ be an $n$-vertex graph with minimum degree at least $\delta n$ for some $\delta := \delta(n) \geq c_1 n^{- 1/3}$. If $\alpha(G) \leq c_2 \delta^2 n$ and $p := p(n) \geq \max \left\{\omega \left(n^{-2} \right),  \frac{c_3 \log(1/\delta)}{\delta n^2} \right\}$, then $G \cup {\mathbb G}(n,p)$ is a.a.s. pancyclic.

\item [(b)] There exists an $n$-vertex graph $H$ with minimum degree $\lfloor n/2 \rfloor - 1$ and independence number $\alpha(H) = 2$ such that if $H \cup {\mathbb G}(n,p)$ is a.a.s. Hamiltonian, then $p := p(n) = \omega\left(n^{-2} \right)$.

\item [(c)] For every $\Omega(n^{- 1/3}) = \delta := \delta(n) = o(1)$ there exists an $n$-vertex graph $H$ with minimum degree at least $\delta n$ satisfying $\alpha(H) = O(\delta^2 n)$ such that a.a.s. $H \cup {\mathbb G}(n,p)$ is not Hamiltonian, whenever $p := p(n) \leq \frac{c_4 \log(1/\delta)}{\delta n^2}$.
\end{description}
\end{main}

Parts (b) and (c) of the statement of Theorem~\ref{th::smallAlpha} show that, given the assumptions $\alpha(G) = O(\delta^2 n)$ and $\delta = \Omega(n^{- 1/3})$, the lower bounds on $p$ stated in Part (a) are essentially best possible. Moreover, given that $\alpha(G) = O(\delta^2 n)$, the assumption $\delta = \Omega(n^{- 1/3})$ is also essentially best possible. Indeed, if $\Delta(G) = O(\delta(G)) = O(\delta n)$, then $\alpha(G) \geq \frac{n}{\Delta(G)+1} = \Omega(1/\delta)$ and thus the assumption $\alpha(G) = O(\delta^2 n)$ implies that $\delta = \Omega(n^{- 1/3})$. In other words, if $\delta = o(n^{- 1/3})$ and $\alpha(G) = O(\delta^2 n)$, then the degree of $(1 - o(1)) n$ of the vertices of $G$ is $\omega(\delta n)$. Nevertheless, considering smaller values of $\delta$ is a worthwhile endeavour. Our next result applies to graphs whose minimum degree is as low as $\Omega \left((\log n)^2 \right)$. 

\begin{main} \label{th::smallDelta}
There exist positive constants $c_1, c_2$ and $c_3$ such that the following hold.
\begin{description}
\item [(a)] Let $G = (V,E)$ be an $n$-vertex graph satisfying $\alpha(G) \delta(G) = O(n)$ and $\delta(G) > c_1 (\alpha(G) \log n)^2$. Then, $G \cup \mathbb{G}(n,p)$ is a.a.s. pancyclic, whenever $p := p(n) \geq \frac{c_2 \log n \log(\alpha(G) \log n)}{n \delta(G)}$.

\item [(b)] For all integers $d := d(n) \geq 1$ and $k := k(n) = \omega_n(1)$ satisfying $k (d+1) \leq n$, there exists an $n$-vertex graph $H$ satisfying $\alpha(H) = k$ and $\delta(H) = d$ such that a.a.s. $H \cup {\mathbb G}(n,p)$ is not Hamiltonian, whenever $p := p(n) \leq \frac{c_3 \log k}{d n}$.
\end{description}
\end{main}

The lower and upper bounds on $p$, stated in parts (a) and (b) of Theorem~\ref{th::smallDelta}, are only a polylogarithmic factor apart. Ignoring polylogarithmic factors, Theorems~\ref{th::smallAlpha}(a) and~\ref{th::smallDelta}(a) together cover the entire range of possible values for the minimum degree of the seed graph $G$. At their ``meeting point'', that is, when $\delta(G) = \tilde{\Theta}(n^{2/3})$ and $\alpha(G) = \tilde{\Theta}(n^{1/3})$, both theorems provide very similar bounds on $p$. 

\bigskip

The assumption $\alpha(G) = O(\delta^2 n)$ appearing in the statement of Theorem~\ref{th::smallAlpha} is essentially the same as the one made in Theorem~\ref{th::BFMalpha}. However, it is not clear that such a restriction on $\alpha(G)$ is required. In fact, the following result suggests that this assumption can be significantly mitigated, provided that Chv\'atal's \emph{toughness conjecture} is true. A graph $G$ is said to be \emph{$t$-tough} if for every $S \subseteq V(G)$ the number of connected components of $G \setminus S$ is at most $\max \{1, |S|/t\}$. Chv\'atal observed that any Hamiltonian graph is 1-tough and, subsequently, made the following intriguing conjecture.

\begin{conjecture} [\cite{Chvatal}] \label{conj::toughness}
There exists a constant $t \geq 1$ such that any $t$-tough graph is Hamiltonian.
\end{conjecture}

Despite considerable effort (see, e.g.,~\cite{BBS} and the many references therein), Conjecture~\ref{conj::toughness} is still open. Our next result reads as follows.

\begin{main} \label{th::toughHam}
\begin{description}
\item [(a)] Suppose that Conjecture~\ref{conj::toughness} holds for some constant $t_0$. Then there exist positive constants $c_1 = c_1(t_0)$ and $c_2 = c_2(t_0)$ such that the following holds. Let $G$ be an $n$-vertex graph with minimum degree $k$. If $\alpha(G) \leq c_2 k$, then $G \cup \mathbb{G}(n,p)$ is a.a.s. pancyclic whenever $p \geq \frac{c_1 \log n}{n k}$.

\item [(b)] For every sufficiently large integer $n$, constant $0 < \varepsilon < 1/2$, integer $n^{\varepsilon} \leq k \leq n^{1 - \varepsilon}$, and constant $0 < c < 1$, there exist a constant $c' > 0$ and an $n$-vertex graph $H$ satisfying $\delta(H) = k$ and $\alpha(H) \leq c k$ such that $H \cup \mathbb{G}(n,p)$ is a.a.s. not Hamiltonian, whenever $p := p(n) \leq \frac{c' \log n}{n k}$.

\item [(c)] For every sufficiently large integer $n$ and every integer $\omega_n(1) = k := k(n) \leq n/2$, there exists an $n$-vertex graph $H$ satisfying $\alpha(H) = k$ and $\delta(H) \geq k/2$ such that $H \cup \mathbb{G}(n,p)$ is a.a.s. not Hamiltonian, whenever $p := p(n) \leq \frac{1}{3 n}$.
\end{description}
\end{main}

Our next result asserts that the assumptions on $\alpha(G)$ and $\delta(G)$ appearing in our previous theorems can be mitigated further (even without assuming the correctness of Conjecture~\ref{conj::toughness}) if one replaces Hamiltonicity with the existence of a nearly spanning cycle (naturally, the bound on $p$ can then be reduced as well).
 
\begin{main} \label{th::longCycle}
\begin{description}
\item [(a)] For every $\varepsilon > 0$ there exist positive constants $c_1, c_2$ and $c_3$ such that the following holds. Let $G$ be an $n$-vertex graph satisfying $\alpha(G) \leq c_1 n$. If $p := p(n) \geq \max \left\{\omega \left(n^{-2} \right),  \frac{c_2 \alpha(G)}{n^2} \right\}$, then a.a.s. $G \cup {\mathbb G}(n,p)$ contains a cycle $C$ of length at least $(1 - \varepsilon) n$. If, additionally, $\alpha(G) \leq c_3 \sqrt{n}$, then a.a.s. $(G \cup {\mathbb G}(n,p))[C]$ is pancyclic.

\item [(b)] For every constant $0 < c \leq 1$ and all integers $c^{-1} < m \leq n$, there exists an $n$-vertex graph $H$ satisfying $\alpha(H) = m$ such that for every $n$-vertex graph $R$ having fewer than $c m/2$ edges, the circumference of $H \cup R$ is less than $c n$.
\end{description}
\end{main}

The rest of this paper is organised as follows. In Section~\ref{sec::prelim} we collect several results that facilitate our proofs. In Section~\ref{sec::smallAlphaPan} we prove Theorems~\ref{th::smallAlpha} and~\ref{th::smallDelta}. Theorems~\ref{th::toughHam} and~\ref{th::longCycle} are proved in Sections~\ref{sec::toughness} and~\ref{sec::cycles}, respectively. Finally, in Section~\ref{sec::concluding}, we discuss possible directions for future research.

\section{Preliminaries} \label{sec::prelim}

In this section we collect several known results, as well as a few new ones, that will be useful in proving our main theorems. In particular, we state and prove a sufficient condition for a graph to admit a spanning system of pairwise disjoint paths whose endpoints have been predetermined (see Theorem~\ref{lem::disjointPaths} in Section~\ref{subsec::linkage}); this result may be of independent interest. Since there are quite a few such results and they vary in nature, this section is divided into several subsections.

\subsection{Matchings and cycles in random graphs and digraphs} \label{subsec::randomGraphs}

Let $\mathbb{D}(n,p)$ denote the probability space of random directed graphs, that is, $D \sim \mathbb{D}(n,p)$ is a digraph with vertex-set $[n] := \{1, \ldots, n\}$ such that, for every $1 \leq i \neq j \leq n$, the \emph{ordered} pair $(i,j)$ is an arc of $D$ with probability $p := p(n)$, independently of all other pairs. The threshold for Hamiltonicity in $\mathbb{D}(n,p)$ is known; the following is sufficient for our purposes.
\begin{theorem} [\cite{Frieze} -- abridged] \label{th::McDiarmid}
If $p \geq (\log n + \omega(1))/n$, then a.a.s. $\mathbb{D}(n,p)$ contains a directed Hamilton cycle.
\end{theorem}

For the emergence of nearly spanning cycles, a smaller probability suffices.

\begin{theorem} [\cite{BKS}, see also~\cite{K}] \label{th::DirectedDFS}
Let $k < n$ be positive integers. Let $G = (V, E)$ be a directed graph on $n$ vertices, such that for any ordered pair of disjoint subsets $S, T \subseteq V$ of size 
$|S| = |T| = k$, $G$ has a directed edge from $S$ to $T$. Then $G$ admits a directed path of length $n - 2k + 1$ and a directed cycle of length at least $n - 4k + 4$.
\end{theorem}

\begin{corollary} \label{cor::AlmostSpanningCycleDnp}
For every $\varepsilon > 0$, there exists a constant $c > 0$ such that $\mathbb{D}(n,p)$ a.a.s. admits a directed cycle of length at least $(1 - \varepsilon) n$, whenever $p := p(n) \geq c/n$.
\end{corollary}

The assertion of Corollary~\ref{cor::AlmostSpanningCycleDnp} is well-known and follows from several fairly old results. In fact, more accurate results are known (see, e.g.,~\cite{KLS}). Nevertheless, for completeness, we include a short simple proof.

\begin{proof} [Proof of Corollary~\ref{cor::AlmostSpanningCycleDnp}]
Let $D \sim \mathbb{D}(n,p)$ and let $k = \varepsilon n/4$. The probability that there exists an ordered pair $(S, T)$ of disjoint subsets of $V(D)$ of size 
$|S| = |T| = k$ such that $E_D(S,T) = \O$ is at most
$$
\binom{n}{k} \binom{n-k}{k} (1 - p)^{k^2} \leq 4^n \exp \left\{- \frac{c}{n} \cdot \frac{\varepsilon^2 n^2}{16} \right\} = o(1),
$$
where the last equality holds for a sufficiently large constant $c := c(\varepsilon)$. It thus follows, by Theorem~\ref{th::DirectedDFS}, that $D$ a.a.s. admits a directed cycle of length at least $n - 4k = (1 - \varepsilon) n$.
\end{proof}

The following theorem is an immediate corollary of a classical result due to Erd\H{o}s and R\'enyi~\cite{ER3}.

\begin{theorem} \label{th::perfectMatchingBnnp}
Let $1 \leq r \leq m$ be integers, and let $G$ be a random subgraph of $K_{r, m}$, obtained by keeping each of its edges independently with probability $p := p(m)$. Then, a.a.s. $G$ admits a matching of size $r$, whenever $p \geq \frac{\log m + \omega_m(1)}{m}$.
\end{theorem}

\subsection{Connectivity, independence and cycles} \label{subsec::cycles}

A graph $G$ is said to be \emph{Hamilton-connected} if, for any two distinct vertices $u, v \in V(G)$, it contains a Hamilton path whose endpoints are $u$ and $v$. The (vertex) connectivity of $G$, denoted $\kappa(G)$, is the smallest size of a set $S \subseteq V(G)$ for which $G \setminus S$ is disconnected (or a single vertex). 
\begin{theorem} [\cite{CE}] \label{th::HamConCE}
If $G$ is a graph satisfying $\kappa(G) > \alpha(G)$, then $G$ is Hamilton-connected.
\end{theorem}

The following two results provide sufficient conditions for the pancyclicity of Hamiltonian graphs.

\begin{theorem} [\cite{KS}] \label{th::deltaVsAlpha}
If $G$ is a Hamiltonian graph with $\delta(G) \geq 600 \alpha(G)$, then $G$ is pancyclic.
\end{theorem}

\begin{theorem} [\cite{DMS}] \label{th::AlphaPancyclic}
There exists a constant $c > 0$ such that for every positive integer $k$, every Hamiltonian graph on $n \geq c k^2$ vertices with $\alpha(G) \leq k$ is pancyclic.
\end{theorem}

\subsection{Connectivity and linkage} \label{subsec::linkage}

The following two results ensure the existence of a highly connected subgraph in a given graph.

\begin{theorem} [\cite{Mader}] \label{th::Mader}
Every graph of average degree at least $k$ admits a $k/4$-connected subgraph.
\end{theorem}

\begin{observation} \label{obs::extractLargeSet}
Let $G = (V,E)$ be an $n$-vertex graph. Then there exists a set $A \subseteq V$ such that $G[A]$ is $\frac{n}{5 \alpha(G)}$-connected.
\end{observation}

\begin{proof}
It follows by Tur\'an's Theorem that the average degree in $G$ is at least $n/\alpha(G) - 1$. It then follows by Theorem~\ref{th::Mader} that there exists a set $A \subseteq V$ such that $G[A]$ is $\frac{n}{5 \alpha(G)}$-connected. 
\end{proof}

The following simple observation allows one to add vertices to a graph while maintaining high connectivity.

\begin{observation} \label{obs::kCon}
Let $G = (V,E)$ be a $k$-connected graph. Let $x \notin V$ and $u_1, \ldots, u_k \in V$ be arbitrary vertices. Then, $G' := (V \cup \{x\}, E \cup \{x u_i : i \in [k]\})$ is $k$-connected.
\end{observation}

The following result allows one to partition a graph into large highly connected subgraphs.

\begin{lemma} [\cite{BFKM}] \label{lem::highConnectivity}
Let $H = (V,E)$ be an $n$-vertex graph with minimum degree $k > 0$. Then, there exists a partition $V = V_1 \cup \ldots \cup V_t$ such that, for every $i \in [t]$, the subgraph $H[V_i]$ is $k^2/(16 n)$-connected and $|V_i| \geq k/8$. 
\end{lemma}

The following result is a variation of Lemma~\ref{lem::highConnectivity} that is better suited for some of our proofs.

\begin{lemma} \label{lem::partition}
Let $G = (V,E)$ be an $n$-vertex graph. There exists a partition $V_1 \cup \ldots \cup V_t$ of $V$ such that the following properties hold.
\begin{description}
\item [(i)] $t \leq 19 \alpha(G) \log n$;

\item [(ii)] $\sum_{i \in J} |V_i| \geq n/2$, where $J := \{i \in [t] : |V_i| \geq 0.1 n/\alpha(G)\}$;

\item [(iii)] $|V_i| \geq \delta(G)/\log n$ for every $i \in [t]$;

\item [(iv)] $\kappa(G[V_i]) \geq \frac{\delta(G)}{20 \alpha(G) \log n}$ for every $i \in [t]$.
\end{description}
\end{lemma}

\begin{proof}
Let $U_1 \subseteq V$ be a set of maximum size such that $G[U_1]$ is $(\delta(G)/\log n)$-connected; such a set exists by Theorem~\ref{th::Mader}. Similarly, let $U_2 \subseteq V \setminus U_1$ be a set of maximum size such that $G[U_2]$ is $(\delta(G)/\log n)$-connected (it is possible that no such set $U_2$ exists, in which case the process ends with $U_1$). Continuing in this manner for as long as possible, let $U_1, \ldots, U_t$ be sets where, for every $i \in [t]$, the set $U_i \subseteq V \setminus (U_1 \cup \ldots \cup U_{i-1})$ is of maximum size such that $G[U_i]$ is $(\delta(G)/\log n)$-connected. 

Let $k$ denote the smallest integer for which $2^{-k} n \leq \alpha(G) \delta(G)$; note that $k \leq \log_2 n$. For every $j \in [k]$, let $I_j = \{i \in [t] : 2^{-j} n < |V \setminus (U_1 \cup \ldots \cup U_{i-1})| \leq 2^{1-j} n\}$; note that $I_j = \O$ is possible. Since $\alpha(G[V \setminus (U_1 \cup \ldots \cup U_{i-1})]) \leq \alpha(G)$ for every $i \in [t]$, it follows by Observation~\ref{obs::extractLargeSet} that $|U_i| \geq \frac{n}{2^j \cdot 5 \alpha(G)}$ holds for every $j \in [k]$ and every $i \in I_j$. Hence,
\begin{equation} \label{eq::ii}
|U_i| \geq \frac{n}{10 \alpha(G)} \textrm{ for every } i \in I_1,
\end{equation}
and $|I_j| \leq \frac{2^{1-j} n \cdot 2^j \cdot 5 \alpha(G)}{n} = 10 \alpha(G)$ for every $j \in [k]$. Moreover, since $|U_i| \geq \delta(G)/\log n$ holds by construction for every $i \in [t]$, it follows by the definition of $k$ that $t - |I_1 \cup \ldots \cup I_k| \leq \alpha(G) \log n$. We conclude that 
\begin{equation} \label{eq::i}
t \leq 10 \alpha(G) k + \alpha(G) \log n \leq 19 \alpha(G) \log n.
\end{equation}

Let $W = V \setminus (U_1 \cup \ldots \cup U_t)$. It follows by the construction of $U_1, \ldots, U_t$ and by Theorem~\ref{th::Mader} that $G[W]$ is $(4 \delta(G)/\log n)$-degenerate; let $w_1, \ldots, w_{\ell}$ be an ordering of the vertices of $W$ such that $\deg_G(w_i, \{w_{i+1}, \ldots, w_{\ell}\}) \leq 4 \delta(G)/\log n$ for every $i \in [\ell]$. Define the required sets $V_1, \ldots, V_t$ by adding the vertices of $W$ to $U_1 \cup \ldots \cup U_t$ one by one as follows. For every $0 \leq j \leq \ell$, define sets $U_1(j), \ldots, U_t(j)$ such that $U_1(j) \cup \ldots \cup U_t(j) = U_1 \cup \ldots \cup U_t \cup \{w_1, \ldots, w_j\}$. For every $i \in [t]$, we start this process with $U_i(0) := U_i$ and end it with $V_i := U_i(\ell)$. Suppose that, for some $j \in [\ell]$, we have already defined $U_1(j-1), \ldots, U_t(j-1)$ and now wish to define $U_1(j), \ldots, U_t(j)$. Let $i \in [t]$ be the smallest index for which $\deg_G(w_j, U_i(j-1)) \geq \frac{\delta(G)}{20 \alpha(G) \log n}$; such an index must exist since $\deg_G(w_j, \{w_{j+1}, \ldots, w_{\ell}\}) \leq 4 \delta(G)/\log n$ and $t \leq 19 \alpha(G) \log n$. Set $U_i(j) = U_i(j-1) \cup \{w_j\}$, and for every $r \in [t] \setminus \{i\}$, set $U_r(j) = U_r(j-1)$.

Note that, by construction, $V_1 \cup \ldots \cup V_t$ is a partition of $V$; in view of~\eqref{eq::ii} and~\eqref{eq::i}, the proof of (i) and (ii) is thus complete. Recalling that $G[U_i]$ is $(\delta(G)/\log n)$-connected for every $i \in [t]$, we note that $|V_i| \geq |U_i| \geq \delta(G)/\log n$ for every $i \in [t]$; this proves (iii). Finally, it follows by Observation~\ref{obs::kCon} that $\kappa(G[V_i]) \geq \frac{\delta(G)}{20 \alpha(G) \log n}$ for every $i \in [t]$; this proves (iv).
\end{proof}


The following result asserts that a random induced subgraph of a graph $G$ inherits, with high probability, some of the connectivity of $G$. 

\begin{theorem} [Theorem 1 in~\cite{CGGHK} -- abridged] \label{th::ConSplit}
Let $G = (V, E)$ be a $k$-connected graph on $n$ vertices, and let $S$ be a randomly sampled subset of $V$, where each vertex $v \in V$ is included in $S$ independently with probability $p := p(n) \geq \alpha \sqrt{\log n/k}$, for a sufficiently large constant $\alpha$. Then, $\kappa(G[S]) = \Omega(k p^2)$ holds with probability $1 - \exp \{- \Omega(k p^2)\}$.
\end{theorem}

Improving earlier results by Robertson and Seymour~\cite{RS} and by Bollob\'as and Thomason~\cite{BT}, it was proved by Thomas and Wollan~\cite{TW} that highly connected graphs are also highly \emph{linked}.

\begin{theorem} [\cite{TW}] \label{th::BT}
Let $G$ be a graph and let $x_1, y_1, \ldots, x_r, y_r$ be $2r$ distinct vertices of $G$. If $\kappa(G) \geq 10 r$, then $G$ admits pairwise vertex-disjoint paths $P_1, \ldots, P_r$ such that, for every $i \in [r]$, the endpoints of $P_i$ are $x_i$ and $y_i$.
\end{theorem}

We prove a \emph{spanning variation} of Theorem~\ref{th::BT}. It may potentially be used to extend a partial embedding of some spanning graph; in the present paper, it is used to extend a matching to a Hamilton cycle.

\begin{theorem} \label{lem::disjointPaths}
Let $G = (V,E)$ be an $n$-vertex graph and let $x_1, y_1, \ldots, x_r, y_r$ be $2r$ distinct vertices of $G$. There exists a constant $c > 0$ such that if $\kappa(G) \geq c \max \{\alpha(G), \log n, r\}$, then $G$ admits pairwise vertex-disjoint paths $P_1, \ldots, P_r$ such that $V(P_1) \cup \ldots \cup V(P_r) = V$ and, for every $i \in [r]$, the endpoints of $P_i$ are $x_i$ and $y_i$.
\end{theorem}

\begin{proof}
Let $G' = G \setminus \{x_1, y_1, \ldots, x_r, y_r\}$; note that $\kappa(G') \geq \kappa(G) - 2r = \Theta(\kappa(G))$, where the equality holds since $\kappa(G) \geq c r$ and $c$ can be chosen to be sufficiently large. Let $S$ be a randomly sampled subset of $V(G')$, where each vertex $v \in V(G')$ is included in $S$ independently with probability $1/2$. Since $\kappa(G) \geq c \log n$ and $c$ can be chosen to be sufficiently large, it follows by Theorem~\ref{th::ConSplit} that a.a.s. $\kappa(G[S]) = \Omega(\kappa(G')) = \Omega(\kappa(G))$. Additionally, by Chernoff's inequality (see, e.g., Corollary 21.7 in~\cite{FKbook}), the probability that there exists a vertex $u \in V(G)$ satisfying $\deg_G(u, S) < (\deg_G(u) - 2r)/3$ is at most $n \exp \{- \Omega(\delta(G) - 2r)\} $; owing to $\delta(G) \geq \kappa(G) \geq c \max \{\log n, r \}$ and to $c$ being sufficiently large, this probability is $o(1)$. The same two claims hold for $V(G') \setminus S$ as well, as it is also a randomly sampled subset of $V(G')$, where each vertex is included in it independently with probability $1/2$. It follows that there exists a partition $S_1 \cup S_2$ of $V(G')$ such that for $j \in \{1,2\}$ we have
\begin{description}
\item [(i)] $\kappa(G[S_j]) = \Omega(\kappa(G))$ and

\item [(ii)] $\deg_G(u, S_j) \geq \kappa(G)/4$ for every $u \in V(G)$.
\end{description}

\begin{figure}[ht] 
	\centering
	\includegraphics[scale = 0.31]{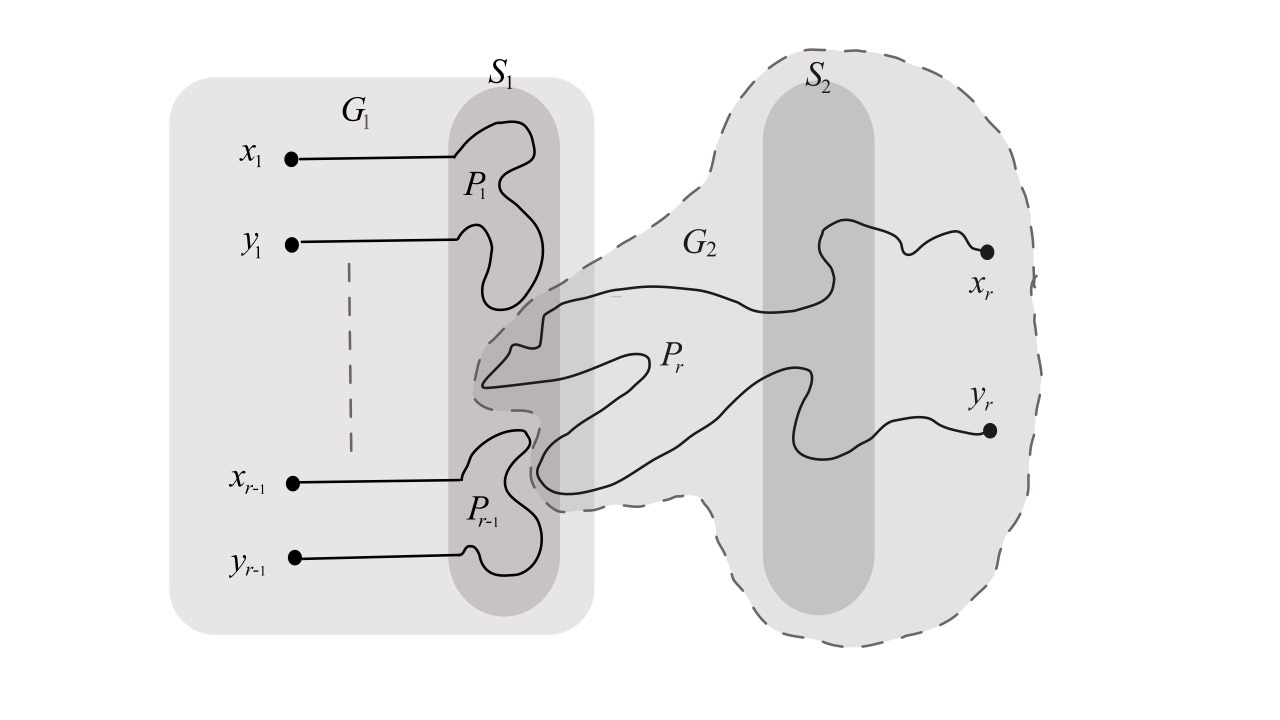}

	\caption{The pairs in $\{\{x_i, y_i\} : i \in [r-1]\}$ are connected via a collection of pairwise disjoint paths in $G_1$. Subsequently, $x_r$ and $y_r$ are connected via a Hamilton path of $G_2$.}
	\label{fig::pathsPartition}
\end{figure}

Set $G_1 = G[S_1 \cup \{x_1, y_1, \ldots, x_{r-1}, y_{r-1}\}]$ and note that $\kappa(G_1) = \Omega(\kappa(G))$ holds by properties (i) and (ii) and by Observation~\ref{obs::kCon}. Moreover, since $\kappa(G) \geq c r$ and $c$ is sufficiently large, it follows by Theorem~\ref{th::BT} that $G_1$ admits pairwise vertex-disjoint paths $P_1, \ldots, P_{r-1}$ such that, for every $i \in [r-1]$, the endpoints of $P_i$ are $x_i$ and $y_i$. Let $G_2 = G \setminus (V(P_1) \cup \ldots \cup V(P_{r-1}))$ and note that $S_2 \cup \{x_r, y_r\} \subseteq V(G_2)$. It then follows by properties (i) and (ii) and by Observation~\ref{obs::kCon}, that $\kappa(G_2) = \Omega(\kappa(G))$. Since, additionally, $\kappa(G) \geq c \alpha(G)$ and $c$ is sufficiently large, it follows by Theorem~\ref{th::HamConCE} that $G_2$ admits a Hamilton path $P_r$ with endpoints $x_r$ and $y_r$ (see Figure~\ref{fig::pathsPartition}).
\end{proof}

\section{Pancyclicity} \label{sec::smallAlphaPan}

\begin{proof} [Proof of Theorem~\ref{th::smallAlpha}]
Starting with (a), let $R \sim {\mathbb G}(n,p)$, where $p := p(n) \geq \max \left\{\omega \left(n^{-2} \right),  \frac{c_3 \log(1/\delta)}{\delta n^2} \right\}$. We first prove that $G \cup R$ is a.a.s. Hamiltonian. If $\delta = \Theta(1)$, then since $p = \omega \left(n^{-2} \right)$, our claim follows by Theorem~\ref{th::BFMalpha}. Assume then that $\delta = o(1)$. Applying Lemma~\ref{lem::highConnectivity} we obtain a partition $V(G) = V_1 \cup \ldots \cup V_t$ such that, for every $i \in [t]$, the subgraph $G[V_i]$ is $(\delta^2 n/16)$-connected and $|V_i| \geq \delta n/8$. 

For every $i \in [t]$, let $A_{i1} \cup \ldots \cup A_{i m_i}$ be an arbitrary partition of $V_i$ such that $\delta n/16 \leq |A_{ij}| \leq \delta n/8$ for every $j \in [m_i]$. Let $m = \sum_{i=1}^t m_i$ denote the total number of sets $A_{ij}$; note that $8/\delta \leq m \leq 16/\delta$. For every $i \in [t]$ and $j \in [m_i]$, let $B^1_{ij} \cup B^2_{ij}$ be an arbitrary partition of $A_{ij}$ satisfying $|B^1_{ij}| \leq |B^2_{ij}| \leq |B^1_{ij}| + 1$. For all distinct pairs $(i_1, j_1) \in [t] \times [m_{i_1}]$ and $(i_2, j_2) \in [t] \times [m_{i_2}]$, it holds that  
\begin{align} \label{eq::oneEdge}
\mathbb{P} \left(E_R \left(B^1_{i_1 j_1}, B^2_{i_2 j_2} \right) \neq \O \right) = 1 - (1-p)^{\left|B^1_{i_1 j_1} \right| \left|B^2_{i_2 j_2} \right|} \geq 1 - \exp \left\{- \frac{p \delta^2 n^2}{33^2}\right\} \geq \frac{2 \log m}{m},
\end{align}
where the last inequality holds since $m \geq 8/\delta = \omega(1)$, $p \geq \frac{c_3 \log(1/\delta)}{\delta n^2}$ for a sufficiently large constant $c_3$, and $e^{-x} \approx 1-x$ holds whenever $x$ tends to $0$.

Consider the auxiliary random directed graph $D$ with vertex-set $\{v_{ij} : i \in [t], j \in [m_i]\}$ such that $(v_{i_1 j_1}, v_{i_2 j_2})$ is an arc of $D$ if and only if $E_R\left(B^1_{i_1 j_1}, B^2_{i_2 j_2} \right) \neq \O$. It follows by~\eqref{eq::oneEdge} and by Theorem~\ref{th::McDiarmid} that $D$ is a.a.s. Hamiltonian; let $C$ be a directed Hamilton cycle of $D$. By construction, for every $i \in [t]$ and every $j \in [m_i]$, there are (two distinct) vertices $x_{ij} \in B^1_{ij}$ and $y_{ij} \in B^2_{ij}$ which correspond to the vertex $v_{ij} \in C$, that is, some edge of $R$ which is incident with $y_{ij}$ corresponds to the arc of $C$ that enters $v_{ij}$ and some edge of $R$ which is incident with $x_{ij}$ corresponds to the arc of $C$ that exits $v_{ij}$.

\begin{figure}[ht] 
	\centering
	\includegraphics[scale = 0.35]{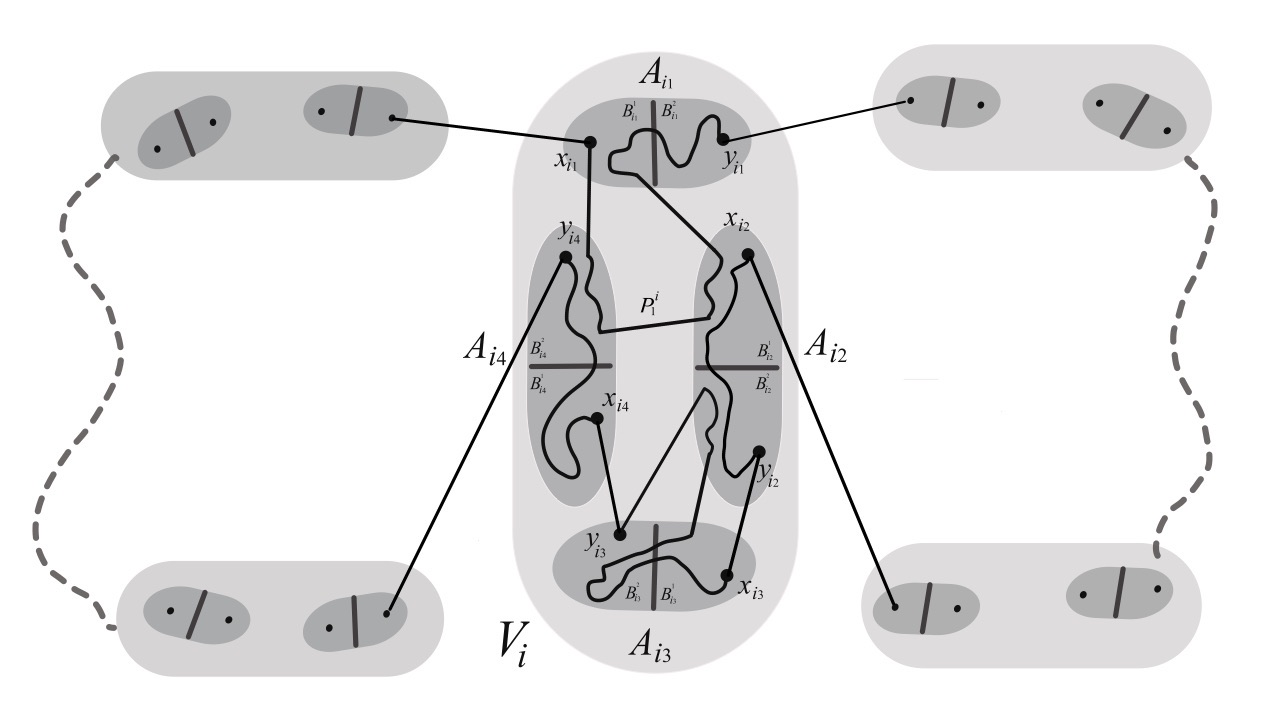}

	\caption{The directed Hamilton cycle $C$ of $D$ is translated into a Hamilton cycle of $G \cup R$. For every $i \in [t]$, the pairs in $\{\{x_{ij}, y_{ij}\} : j \in [m_i]\}$ are connected via a collection of pairwise disjoint paths that span $V_i$.}
	\label{fig::HamCycle}
\end{figure}

Let $c$ be the constant whose existence is ensured by Theorem~\ref{lem::disjointPaths}. Fix any $i \in [t]$. Note that $\kappa(G[V_i]) \geq \delta^2 n/16 \geq c \max \{\alpha(G[V_i]), \log n, m_i\}$ holds by the premise of the theorem for appropriately chosen constants $c_1$ and $c_2$, and since $m_i \leq m \leq 16/\delta$ and $\delta \geq c_1 n^{- 1/3}$. It thus follows by Theorem~\ref{lem::disjointPaths} that, for every $i \in [t]$, there are pairwise vertex-disjoint paths $P^i_1, \ldots, P^i_{m_i}$ such that $V(P^i_1) \cup \ldots \cup V(P^i_{m_i}) = V_i$ and, for every $j \in [m_i]$, the endpoints of $P^i_j$ are $x_{ij}$ and $y_{ij}$. Replacing every vertex $v_{ij}$ in $C$ with the corresponding path $P^i_j$ and replacing every arc of $C$ with the corresponding edge of $R$ yields a Hamilton cycle of $G \cup R$ (see Figure~\ref{fig::HamCycle}). Finally, since $\alpha(G \cup R) \leq \alpha(G) \leq \delta n/600 \leq \delta(G \cup R)/600$ holds for sufficiently small $c_2$, it follows by Theorem~\ref{th::deltaVsAlpha} that $G \cup R$ is in fact a.a.s. pancyclic.


\bigskip

Next, we prove (b). Let $H$ be the disjoint union of $K_{\lfloor n/2 \rfloor}$ and $K_{\lceil n/2 \rceil}$. Note that $\delta(H) = \lfloor n/2 \rfloor - 1$ and $\alpha(H) = 2$. Moreover, for any constant $c > 0$, if $p = c n^{-2}$, then the probability that $R \sim {\mathbb G}(n,p)$ has no edges is bounded away from 0. Hence, the probability that $H \cup R = H$ is disconnected and thus, in particular, not Hamiltonian, is bounded away from 0.

\bigskip

Finally, we prove (c). Let $\Omega(n^{- 1/3}) = \delta := \delta(n) = o(1)$. Let $H$ be an $n$-vertex graph consisting of $k = (1 + o(1)) \delta^{-1}$ pairwise disjoint cliques $Q_1, \ldots, Q_k$ satisfying $\delta n + 1 \leq |V(Q_1)| \leq \ldots \leq |V(Q_k)| \leq |V(Q_1)| + 1$. It follows that $\delta(H) \geq \delta n$ and that $\alpha(H) = k = (1 + o(1)) \delta^{-1} = O(\delta^2 n)$. Let $R \sim {\mathbb G}(n,p)$, where $p := p(n) \leq \frac{c_4 \log(1/\delta)}{\delta n^2}$. Consider the auxiliary random graph $G = (V,E)$, where $V = \{u_1, \ldots, u_k\}$ and $u_i u_j \in E$ if and only if $E_R(V(Q_i), V(Q_j)) \neq \O$. Hence, for every two distinct indices $i, j \in [k]$, the probability that $u_i u_j \in E$ is 
$$
1 - (1-p)^{|V(Q_i)||V(Q_j)|} \leq 1 - \left(1 - p |V(Q_i)||V(Q_j)| \right) \leq \frac{c_4 \log(1/\delta)}{\delta n^2} (2 \delta n)^2 \leq \frac{\log k}{2k},
$$
where the first inequality holds by Bernoulli's inequality, the second inequality holds since $p \leq \frac{c_4 \log(1/\delta)}{\delta n^2}$, and the last inequality holds for a sufficiently small constant $c_4$ since $\delta = (1 + o(1)) k^{-1}$ and $k = \omega(1)$. It thus follows by a classical result of Erd\H{o}s and R\'enyi~\cite{ER} that a.a.s. $G$ admits an isolated vertex. It follows, by construction, that a.a.s. there exists some $i \in [k]$ such that $E_R(V(Q_i), V(H) \setminus V(Q_i)) = \O$. We conclude that a.a.s. $H \cup R$ is disconnected and, in particular, not Hamiltonian.   
\end{proof}

\begin{proof} [Proof of Theorem~\ref{th::smallDelta}]
Starting with (a), Let $c$ be the constant whose existence is ensured by Theorem~\ref{lem::disjointPaths}. Let $V_1 \cup \ldots \cup V_t$ be a partition of $V$ as in the statement of Lemma~\ref{lem::partition}. Let 
$$
S = \left\{i \in [t] : |V_i| < \frac{n}{100 \alpha(G) \log n} \right\}.
$$ 
If $S \neq \O$, then assume without loss of generality (by relabelling) that $S = [s]$ for some $s \in [t]$. For every $i \in [s]$, let $B_i^1 \cup B_i^2$ be a partition of $V_i$ such that $|B_i^1| \leq |B_i^2| \leq |B_i^1| + 1$. For every $i \in [t] \setminus [s]$, let $A_i^1 \cup \ldots \cup A_i^{m_i}$ be an arbitrary equipartition (that is, $|A_i^1| \leq \ldots \leq |A_i^{m_i}| \leq |A_i^1| + 1$) of $V_i$, where $m_i$ is the smallest positive integer for which $|A_i^j| \leq \frac{n}{100 \alpha(G) \log n}$ for every $j \in [m_i]$. Note that $|A_i^j| \geq \frac{n}{200 \alpha(G) \log n}$ for every $i \in [t] \setminus [s]$ and every $j \in [m_i]$; in particular, $m_i \leq 200 \alpha(G) \log n$ for every $i \in [t] \setminus [s]$. Let $L = \{A_i^j : i \in [t] \setminus [s], j \in [m_i]\}$.


Let $q = 2 s$ and let $X_1, \ldots, X_q$ be an ordering of the sets in $\{B_i^k : i \in [s], k \in \{1,2\}\}$ such that $X_{2i-1} = B_i^1$ and $X_{2i} = B_i^2$ for every $i \in [s]$. Note that $|X_i| \geq \delta(G)/(2 \log n)$ holds for every $i \in [q]$ by Lemma~\ref{lem::partition}(iii). Let $Y_1, \ldots, Y_{q'}$ be an arbitrary ordering of the sets in $L$; note that $|Y_i| \geq \frac{n}{200 \alpha(G) \log n}$ holds for every $i \in [q']$. Additionally, $q' \geq 50 \alpha(G) \log n \geq 2t \geq q$ holds by the definitions of $S$ and $L$ and by assertions (i) and (ii) of Lemma~\ref{lem::partition}. Let $R \sim \mathbb{G}(n,p)$, where $p := p(n) \geq \frac{c_2 \log n \log(\alpha(G) \log n)}{n \delta(G)}$ for a sufficiently large constant $c_2$. Exposing first only the edges of $R$ with one endpoint in $\bigcup_{i=1}^q X_i$ and the other in $\bigcup_{j=1}^{q'} Y_j$, consider the auxiliary random bipartite graph $H$ with parts $X = \{x_1, \ldots, x_q\}$ and $Y = \{y_1, \ldots, y_{q'}\}$, where for every $i \in [q]$ and $j \in [q']$ there is an edge of $H$ connecting $x_i$ and $y_j$ if and only if $E_R(X_i, Y_j) \neq \O$. For every $i \in [q]$ and $j \in [q']$ it holds that
$$
\mathbb{P}(E_R(X_i, Y_j) \neq \O) = 1 - (1-p)^{|X_i||Y_j|} \geq 1 - \exp \left\{- p \frac{n \delta(G)}{400 \alpha(G) (\log n)^2} \right\} \geq \frac{2 \log q'}{q'},
$$
where the last inequality holds since $q' \geq 50 \alpha(G) \log n$, $p \geq \frac{c_2 \log n \log(\alpha(G) \log n)}{n \delta(G)}$ for a sufficiently large constant $c_2$, and $e^{-x} \approx 1-x$ holds whenever $x$ tends to $0$. Since, moreover, $q' \geq q$, it follows by Theorem~\ref{th::perfectMatchingBnnp} that a.a.s. $H$ admits a matching of size $q$; conditioning on this event, assume without loss of generality that ${\mathcal M} = \{x_i y_i : i \in [q]\}$ is such a matching. For every $i \in [q]$, let $u_i \in X_i$ and $v_i \in Y_i$ be vertices for which $u_i v_i \in E(R)$.

\begin{figure}[ht] 
	\centering
	\includegraphics[scale = 0.36]{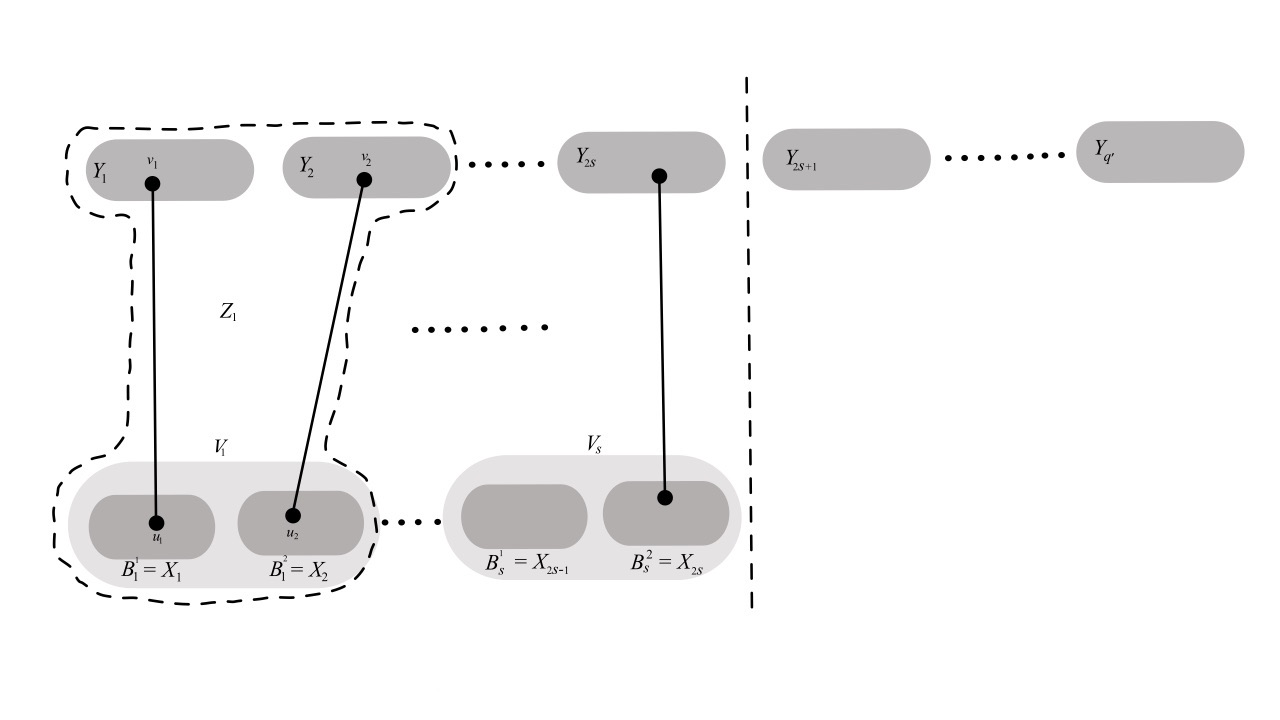}

	\caption{The matching ${\mathcal M}$, the vertices $u_i$ and $v_i$, and the sets $Z_i$ for $i \in [s]$.}
	\label{fig::matching}
\end{figure}

Let $\ell = q' - s$. For every $i \in [\ell]$, define sets $Z_i, L_i$ and $R_i$ as follows. If $i \in [s]$, then $Z_i := Y_{2i-1} \cup X_{2i-1} \cup X_{2i} \cup Y_{2i}$ (see Figure~\ref{fig::matching}), $L_i := Y_{2i-1} \setminus \{v_{2i-1}\}$, and $R_i := Y_{2i} \setminus \{v_{2i}\}$; otherwise $Z_i := Y_{i + s}$ and $L_i \cup R_i$ is an arbitrary partition of $Z_i$ such that $|L_i| \leq |R_i| \leq |L_i| + 1$. Exposing the remaining edges of $R$, for every two distinct indices $i, j \in [\ell]$, it holds that
\begin{align} \label{eq::oneEdgeDHam}
\mathbb{P}(E_R(R_i, L_j) \neq \O) = 1 - (1-p)^{|R_i||L_j|} \geq 1 - \exp \left\{-p \left(\frac{n}{400 \alpha(G) \log n} \right)^2 \right\} \geq \frac{2 \log \ell}{\ell},
\end{align}
where the last inequality holds since $\ell \geq 30 \alpha(G) \log n$, $\alpha(G) \delta(G) = O(n)$, $p \geq \frac{c_2 \log n \log(\alpha(G) \log n)}{n \delta(G)}$ for a sufficiently large constant $c_2$, and $e^{-x} \approx 1-x$ holds whenever $x$ tends to $0$. 

Consider the auxiliary random directed graph $D$ with vertex-set $\{z_i : i \in [\ell]\}$ such that $(z_i, z_j)$ is an arc of $D$ if and only if $E_R\left(R_i, L_j \right) \neq \O$. It follows by~\eqref{eq::oneEdgeDHam} and by Theorem~\ref{th::McDiarmid} that $D$ is a.a.s. Hamiltonian; let $C$ be a directed Hamilton cycle of $D$. By construction, for every $i \in [\ell]$ there are vertices $a_i \in R_i$ and $b_i \in L_i$ which correspond to the vertex $z_i \in C$, that is, some edge of $R$ which is incident with $b_i$ corresponds to the arc of $C$ that enters $z_i$ and some edge of $R$ which is incident with $a_i$ corresponds to the arc of $C$ that exits $z_i$.

For every $i \in [s]$, it follows by Lemma~\ref{lem::partition}(iv) and by the premise of the theorem that $\kappa(G[V_i]) \geq \frac{\delta(G)}{20 \alpha(G) \log n} > \alpha(G) \geq \alpha(G[V_i])$. It thus follows by Theorem~\ref{th::HamConCE} that for every $i \in [s]$ the graph $G[V_i]$ admits a Hamilton path $Q_i$ whose endpoints are $u_{2i-1}$ and $u_{2i}$. 

Next, fix an arbitrary $i \in [t] \setminus [s]$. For every $j \in [m_i]$, let $r_j \in [q']$ be the unique integer for which $A_i^j = Y_{r_j}$. Note that $\kappa(G[V_i]) \geq \frac{\delta(G)}{20 \alpha(G) \log n} \geq c \max \{\alpha(G), \log n, m_i\}$ holds by Lemma~\ref{lem::partition}(iv), by the premise of the theorem for a sufficiently large constant $c_1$, and since $m_i \leq 200 \alpha(G) \log n$. It thus follows by Theorem~\ref{lem::disjointPaths} that there are pairwise vertex-disjoint paths $P_i^1, \ldots, P_i^{m_i}$ such that $V(P_i^1) \cup \ldots \cup V(P_i^{m_i}) = V_i$, and for every $j \in [m_i]$, the endpoints of $P_i^j$ are 
\begin{description}
\item [(1)] $a_{r_j}$ and $v_{r_j}$ if $r_j \in [q]$ is even.

\item [(2)] $b_{r_j}$ and $v_{r_j}$ if $r_j \in [q]$ is odd.

\item [(3)] $a_{r_j}$ and $b_{r_j}$ if $r_j \in [q'] \setminus [q]$.
\end{description}
Replacing every vertex $z_i$ in $C$, where $i \in [q]$, with the corresponding path 
$$
b_{2i-1} P^c_d v_{2i-1} u_{2i-1} Q_i u_{2i} v_{2i} P^e_f a_{2i}
$$
(see Figure~\ref{fig::pathCompositeVertex}), replacing every vertex $z_i$ in $C$, where $i \in [q'] \setminus [q]$, with the corresponding path $P^r_j$, and replacing every arc of $C$ with the corresponding edge of $R$, yields a Hamilton cycle of $G \cup R$.

\begin{figure}[ht] 
	\centering
	\includegraphics[scale = 0.35]{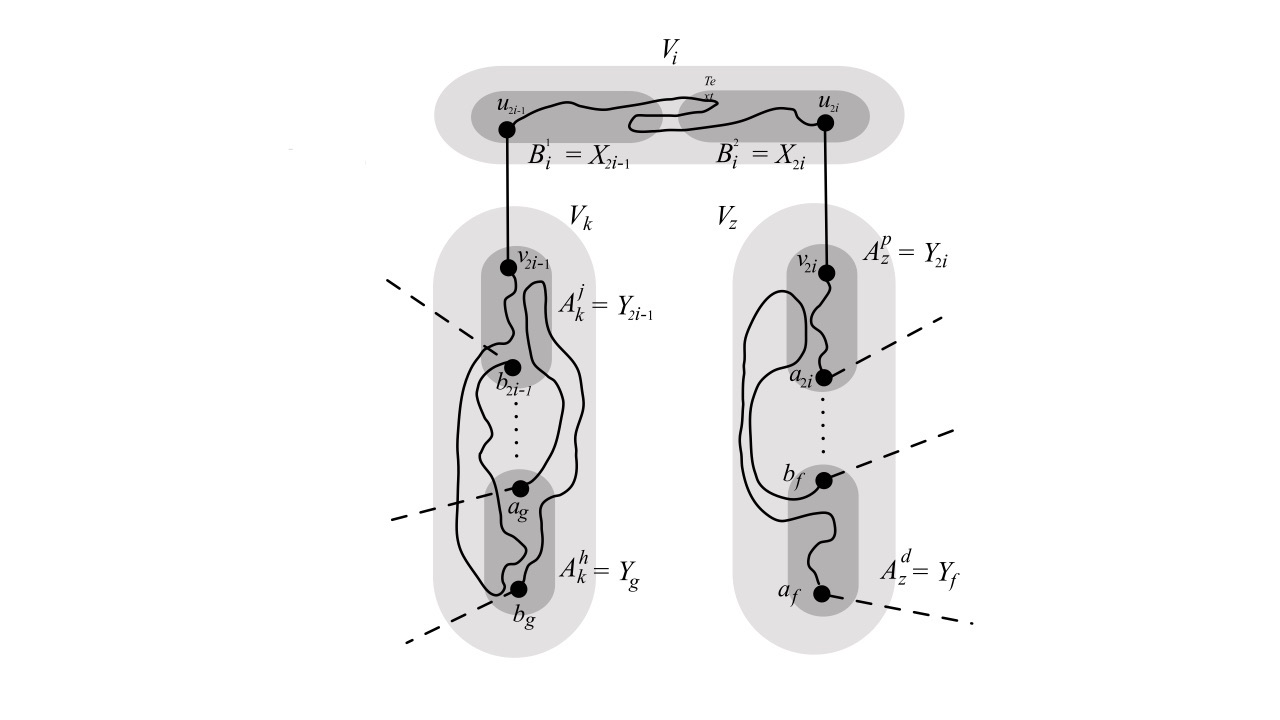}

	\caption{The Hamilton path of $G[V_i]$, where $i \in [s]$, with endpoints $u_{2i-1}$ and $u_{2i}$ is $Q_i$. Moreover, $Y_{2i-1} \subseteq V_k$ and $Y_{2i} \subseteq V_z$ for some, not necessarily distinct, indices $k, z \in [t] \setminus [s]$. }
	\label{fig::pathCompositeVertex}
\end{figure}

Finally, since $\alpha(G \cup R) \leq \alpha(G) = o(\delta(G)) = o(\delta(G \cup R))$, it follows by Theorem~\ref{th::deltaVsAlpha} that $G \cup R$ is in fact a.a.s. pancyclic. 

\bigskip


Next, we prove (b). Let $d := d(n) \geq 1$ and $k := k(n) = \omega_n(1)$ be integers satisfying $k (d+1) \leq n$. Let $H$ be a graph consisting of $k$ pairwise disjoint cliques $Q_1, \ldots, Q_k$ satisfying $|V(Q_1)| = \ldots = |V(Q_{k-1})| = d + 1$ and $|V(Q_k)| = n - (k-1)(d+1) \geq d+1$. It follows that $|V(H)| = n$, $\delta(H) = d$, and $\alpha(H) = k$. Let $R \sim {\mathbb G}(n,p)$, where $p \leq \frac{c_3 \log k}{d n}$. Consider the auxiliary random graph $G = (V,E)$, where $V = \{u_1, \ldots, u_k\}$ and $u_i u_j \in E$ if and only if $E_R(V(Q_i), V(Q_j)) \neq \O$. For every $j \in [k-1]$, let $I_j$ be the indicator random variable for the event ``$u_j$ is isolated in $G$'' and let $X = \sum_{j=1}^{k-1} I_j$. Standard calculations show that $\mathbb{E}(X) = \omega_n(1)$ and that $\textrm{Var}(X) = o \left((\mathbb{E}(X))^2 \right)$. It thus follows by the second moment method that a.a.s. $G$ admits an isolated vertex. It then follows, by construction, that a.a.s. there exists some $i \in [k-1]$ such that $E_R(V(Q_i), V(H) \setminus V(Q_i)) = \O$. We conclude that a.a.s. $H \cup R$ is disconnected and, in particular, not Hamiltonian.  
\end{proof}

\section{Toughness} \label{sec::toughness}

\begin{proof} [Proof of Theorem~\ref{th::toughHam}]

Starting with (a), we expose $R \sim \mathbb{G}(n, p)$ in two rounds, that is, $R = R_1 \cup R_2$, where $R_i \sim \mathbb{G}(n, q)$ for $i \in \{1,2\}$ and $q$ satisfying $(1-q)^2 = 1-p$; note that $q \geq p/2$. Assuming that $c_2 \leq 1/600$, it follows by Theorem~\ref{th::deltaVsAlpha} that if $G \cup R$ is Hamiltonian, then it is in fact pancyclic. Hence, it remains to prove that $G \cup R$ is a.a.s. Hamiltonian. Our aim is to prove that $G \cup R$ is a.a.s. $t_0$-tough, that is, that a.a.s., for every $S \subseteq V(G)$, the number of connected components of $(G \cup R) \setminus S$ is at most $\max \{1, |S|/t_0\}$.  Let $S \subseteq V(G)$ be an arbitrary set, and let $r$ denote the number of connected components of $G \setminus S$; note that the number of connected components of $(G \cup R) \setminus S$ is at most $r$. If $|S| \geq k/2$, then $r \leq \alpha(G \setminus S) \leq \alpha(G) \leq c_2 k \leq |S|/t_0$, where the last inequality holds for $c_2 \leq \frac{1}{2 t_0}$. Assume then that $|S| \leq k/2$. Since $\delta(G) = k$, it follows that $\delta(G \setminus S) \geq k/2$. Hence $r \leq \frac{n - |S|}{\delta(G \setminus S) + 1} \leq \frac{2n}{k}$. We may thus assume that $|S| \leq \frac{2 n t_0}{k}$. Fix an arbitrary integer $1 \leq s \leq 2 n t_0/k$ and an arbitrary subset $S \subseteq V(G)$ of size $s$.

\medskip

Expose $R_1$. Let $C_1, \ldots, C_r$ denote the connected components of $G \setminus S$, where $|V(C_1)| \leq \ldots \leq |V(C_r)|$. We claim that with very high probability (properly quantified below) $(G \cup R_1) \setminus S$ admits a connected component spanning at least $n/3$ vertices. Indeed, otherwise $|V(C_1)| \leq \ldots \leq |V(C_r)| \leq n/3$ and thus there exists an index $j \in [r]$ such that $\sum_{i=1}^j |V(C_i)| \geq n/3$ and $\sum_{i=j+1}^r |V(C_i)| \geq n/3$. Hence, the probability that the order of every connected component of $(G \cup R_1) \setminus S$ is smaller than $n/3$ is at most 
\begin{align} \label{eq::largeComp}
(1-q)^{n/3 \cdot n/3} = o \left(n \binom{n}{s} \right)^{-1},
\end{align}
where the equality holds since $s \leq 2 n t_0/k$ and $q \geq p/2 \geq \frac{c_1 \log n}{2 n k}$ for a sufficiently large constant $c_1 := c_1(t_0)$.

\medskip

Suppose now that $(G \cup R_1) \setminus S$ has a connected component $C$ such that $|V(C)| \geq n/3$. Expose $R_2$. If $(G \cup R) \setminus S$ has more than $r' := \max \{1, s/t_0\}$ connected components, then there are connected components $Q_1, \ldots, Q_{r'}$ of $(G \cup R) \setminus (S \cup C)$ such that $E_{R_2}(V(C), V(Q_1) \cup \ldots \cup V(Q_{r'})) = \emptyset$. The probability of this event is at most
\begin{align} \label{eq::manyComp}
\binom{r}{r'} (1-q)^{n/3 \cdot \sum_{i=1}^{r'} |V(Q_i)|} \leq n^s \exp \left\{- q \cdot \frac{n}{3} \cdot \frac{s}{t_0} \cdot \frac{k}{2}\right\} = o \left(n \binom{n}{s} \right)^{-1},
\end{align}
where the inequality holds since any connected component of $G \setminus S$ (and therefore also of $(G \cup R) \setminus S$) is of order at least $k/2$ and the equality holds since $s \leq 2 n t_0/k$ and $q \geq p/2 \geq \frac{c_1 \log n}{2 n k}$ for a sufficiently large constant $c_1 := c_1(t_0)$.

\medskip

Combining~\eqref{eq::largeComp} and~\eqref{eq::manyComp} we conclude that the probability that there exists a subset $S \subseteq V(G)$ for which the number of connected components of $(G \cup R) \setminus S$ is larger than $\max \{1, |S|/t_0\}$ is at most
$$
\sum_{s=1}^{2 n t_0/k} \binom{n}{s}\left[o \left(\left(n \binom{n}{s} \right)^{-1} \right) + o \left( \left(n \binom{n}{s} \right)^{-1} \right) \right] = o(1).
$$

\medskip

Next, we prove (b). Let $H$ be an $n$-vertex graph consisting of $r := \min \{\lfloor c k \rfloor, \lfloor n/(k+1) \rfloor\} - 1$ pairwise disjoint cliques $Q_1, \ldots, Q_r$ of size $|V(Q_1)| = \ldots = |V(Q_r)| = k+1$ and one additional clique with vertex-set $V(H) \setminus (V(Q_1) \cup \ldots \cup V(Q_r))$. Note that $\delta(H) = k$ and $\alpha(H) \leq c k$. Let $R \sim {\mathbb G}(n,p)$, where $p \leq \frac{c' \log n}{n k}$. For every $i \in [r]$, let $I_i$ be the indicator random variable for the event ``$E_R(V(Q_i), V(H) \setminus V(Q_i)) = \O$'', and let $X = \sum_{i=1}^r I_i$. Standard calculations show that $\mathbb{E}(X) = \omega_n(1)$ and that $\textrm{Var}(X) = o \left((\mathbb{E}(X))^2 \right)$. It thus follows by the second moment method that a.a.s. $E_R(V(Q_i), V(H) \setminus V(Q_i)) = \O$ for some $i \in [r]$. We conclude that a.a.s. $H \cup R$ is disconnected and, in particular, not Hamiltonian.

\medskip

Finally, we prove (c). Let $V$ be a set of size $n$ and let $I \cup A \cup B$ be a partition of $V$ such that $|I| = k-1$ and $|A| = \lceil k/2 \rceil$. Let $H = (V, E)$ be a graph, where 
$$
E = \{xy : x \in I \textrm{ and } y \in A\} \cup \{xy : x \neq y \textrm{ and } \{x,y\} \subseteq A \cup B\}.
$$
Observe that $\alpha(H) = |I| + 1 = k$ and $\delta(H) = |A| \geq k/2$. Deleting $A$ partitions $H[I \cup B]$ into $k$ connected components. It follows that if $R$ is any graph with vertex-set $V$ such that $H \cup R$ is Hamiltonian and thus, in particular, 1-tough, then $r := |\{u \in I : \exists v \in I \cup B \textrm{ such that } uv \in E(R)\}| \geq \lfloor k/2 \rfloor$. However, a straightforward calculation shows that a.a.s. $r < \lfloor k/2 \rfloor$ whenever $R \sim \mathbb{G}(n,p)$ for $p \leq \frac{1}{3 n}$.
\end{proof}

\begin{remark}
While Conjecture~\ref{conj::toughness}, on which the assertion of Theorem~\ref{th::toughHam}(a) relies, is still open, it is well-known (see, e.g.,~\cite{BBS}) that any $t$-tough graph on $n \geq t+1$ vertices, where $t n$ is even, admits a $t$-factor. In particular, it thus follows from our proof of Theorem~\ref{th::toughHam}, that $G \cup \mathbb{G}(n,p)$ a.a.s. admits a perfect matching, whenever $G$ is an $n$-vertex graph for even $n$, $\alpha(G) \leq c \delta(G)$ for a sufficiently small constant $c > 0$, and $p \geq \frac{c' \log n}{n \delta(G)}$ for a sufficiently large constant $c'$. A straightforward adaptation of our proof of Theorem~\ref{th::toughHam}(b) (taking all cliques but at most one to be of odd order) shows that the aforementioned bound on $p$ is essentially best possible. A straightforward adaptation of the proof of Theorem~\ref{th::toughHam}(c) shows that the aforementioned bound on $\alpha(G)$ is essentially best possible.
\end{remark}

\section{Long cycles} \label{sec::cycles}

\begin{proof} [Proof of Theorem~\ref{th::longCycle}]
Starting with (a), let $m = \alpha(G)$. Assume first that $m = O(1)$, and let $G'$ be the graph obtained from $G$ by repeatedly discarding vertices of degree at most $\delta n$, where $\delta := \varepsilon/(2m)$. Then, $\delta(G') \geq \delta n$ and $\alpha(G') \leq m$. Moreover, $|V(G')| \geq (1 - \varepsilon) n$. Indeed, let $S = V(G) \setminus V(G')$ and assume for a contradiction that $|S| \geq \varepsilon n$. Observe that $G[S]$ is $\delta n$-degenerate, implying that 
$$
m = \alpha(G) \geq \alpha(G[S]) \geq \frac{|S|}{\delta n + 1} \geq \frac{\varepsilon n}{\varepsilon n/(2m) + 1} > m,
$$ 
which is a contradiction. Since $p = \omega \left(n^{-2} \right)$, it follows by Theorem~\ref{th::smallAlpha} that a.a.s. $(G \cup \mathbb{G}(n,p))[V(G')]$ is pancyclic. 

We may thus assume for the remainder of the proof that $m = \omega(1)$. Let $\ell = \varepsilon n/(4m)$ and let $r = 4 (1 - \varepsilon/3) \varepsilon^{-1} m$. We claim that $G$ admits pairwise vertex-disjoint paths $P_1, \ldots, P_r$, each of length $\ell - 1$. Indeed, assuming that for some $i \in [r]$, the paths $P_1, \ldots, P_{i-1}$ have already been constructed, the path $P_i$ is found as follows. Let $G_i = G \setminus (V(P_1) \cup \ldots \cup V(P_{i-1}))$, so that $|V(G_i)| = n - (i-1) \ell \geq n - r \ell = \varepsilon n/3$. Let $G'_i$ be the graph obtained from $G_i$ by repeatedly discarding vertices of degree at most $\ell - 1$. Note that $\alpha(G'_i) \leq m$. Moreover, $G'_i$ is non-empty and thus $\delta(G'_i) \geq \ell$ holds by construction. Indeed, let $S = V(G_i) \setminus V(G'_i)$ and assume for a contradiction that $|S| \geq \varepsilon n/3$. Observe that $G_i[S]$ is $(\ell-1)$-degenerate, implying that 
$$
m = \alpha(G) \geq \alpha(G_i[S]) \geq \frac{|S|}{\ell} \geq \frac{\varepsilon n/3}{\varepsilon n/(4m)} > m,
$$ 
which is a contradiction. It follows that $G'_i$ contains a path of length $\ell - 1$ which we denote by $P_i$.

For every $i \in [r]$, let $X_i$ denote the set of the first $\varepsilon \ell/6$ vertices of $P_i$ (according to an arbitrary consistent orientation) and let $Y_i$ denote the set of the last $\varepsilon \ell/6$ vertices of $P_i$. Let $R \sim \mathbb{G}(n,p)$, where $p := p(n) \geq \frac{c_2 m}{n^2}$. Consider the auxiliary random directed graph $D$, where $V(D) = \{u_1, \ldots, u_r\}$ and $(u_i, u_j) \in E(D)$ if and only if $E_R(Y_i, X_j) \neq \O$. For every two distinct indices $i, j \in [r]$, the probability that $(u_i, u_j) \in E(D)$ is 
$$
1 - (1-p)^{|Y_i||X_j|} \geq 1 - \exp \left\{- \varepsilon^2 \ell^2 p/40 \right\} \geq 1 - \exp \left\{- c_2 \varepsilon^4/(640 m) \right\} \geq \frac{c_2 \varepsilon^4}{650 m} \geq \frac{c_2 \varepsilon^3}{250 r},
$$
where the second inequality holds since $p \geq \frac{c_2 m}{n^2}$ and $\ell = \varepsilon n/(4m)$ and the penultimate inequality holds since $m = \omega(1)$ and since $e^{-x} \approx 1-x$ holds whenever $x$ tends to $0$. 

For a sufficiently large constant $c_2$, it follows by Corollary~\ref{cor::AlmostSpanningCycleDnp}, that $D$ contains a directed cycle $C$ of length at least $(1 - \varepsilon/3) r$; assume without loss of generality that $u_1, \ldots, u_t$ are the vertices of $C$. By construction, for every $i \in [t]$, there are vertices $x_i \in X_i$ and $y_i \in Y_i$ which correspond to the vertex $u_i\in C$, that is, some edge of $R$ which is incident with $x_i$ corresponds to the arc of $C$ that enters  $u_i$ and some edge of $R$ which is incident with $y_i$ corresponds to the arc of $C$ that exits $u_i$. Replacing every vertex $u_i$ of $C$ with the corresponding subpath of $P_i$ connecting $x_i$ and $y_i$ and replacing every arc of $C$ with the corresponding edge of $R$ yields a cycle $C'$ of $G \cup R$ of length at least $t (1 - \varepsilon/3) \ell \geq (1 - \varepsilon) n$. Finally, note that if, in addition, $\alpha(G) \leq c_3 \sqrt{n}$, then $(G \cup R)[C']$ is pancyclic by Theorem~\ref{th::AlphaPancyclic}.

\bigskip

Next, we prove (b). Let $H$ be an $n$-vertex graph consisting of $m$ pairwise disjoint cliques $Q_1, \ldots, Q_m$ satisfying $|V(Q_1)| \leq \ldots \leq |V(Q_m)| \leq |V(Q_1)| + 1$. It is evident that in order to obtain a cycle of length at least $c n$, one has to add at least $\frac{c n}{\lceil n/m \rceil} \geq c m/2$ edges to $H$.
\end{proof}

\section{Possible directions for future research} \label{sec::concluding}

As noted in the introduction, it is suggested in~\cite{BFM} that the reason for $K_{n/3, 2n/3}$ requiring linearly many additional edges to become Hamiltonian is that it admits a large independent set; this observation has led to the statement and proof of Theorem~\ref{th::BFMalpha}. However, there may be other explanations for this phenomenon. For example, regularity is sometimes beneficial for Hamiltonicity (see, e.g.,~\cite{KO} and  references therein) and $K_{n/3, 2n/3}$ is far from being regular in the sense that its maximum degree is significantly larger than its minimum degree. The following result asserts that regular graphs can be made Hamiltonian by the addition of relatively few random edges.


\begin{theorem} \label{th::DenseRegular}
Let $G$ be an $r$-regular graph on $n$ vertices. Then, there exists a constant $c > 0$ such that $G \cup {\mathbb G}(n,p)$ is a.a.s. Hamiltonian, whenever $p := p(n) \geq \min \left\{\frac{(1 + o(1)) \log n}{n}, \frac{c \sqrt{n^3 \log n}}{r^3} \right\}$.
\end{theorem}

Our proof of Theorem~\ref{th::DenseRegular} makes use of the following result which we believe to be of independent interest.
\begin{theorem} \label{th::FewCyclesExpansion} 
Let $G$ be a $(\gamma n, 2)$-expander\footnote{A graph $G$ is said to be a \emph{$(k,d)$-expander} if $|N_G(S)| \geq d |S|$ holds for every subset $S \subseteq V(G)$ of size at most $k$.} on $n$ vertices, for some $\gamma := \gamma(n)$. Let $C_1, \ldots, C_k$, where $k := k(n)$, be pairwise vertex-disjoint cycles in $G$ and let $\ell = |V(G) \setminus (V(C_1) \cup \ldots \cup V(C_k))|$. Then, $G \cup {\mathbb G}(n,p)$ is a.a.s. Hamiltonian, whenever 
$$
p := p(n) \geq 
\begin{cases}
\omega \left((\gamma n)^{-2} \right) & \emph{if } \; k + \ell = O(1), \\
\frac{c (k + \ell)}{\gamma^2 n^2} & \emph{if } \; k + \ell = \omega(1),
\end{cases}
$$
where $c$ is a sufficiently large constant.
\end{theorem}

One may prove Theorem~\ref{th::FewCyclesExpansion} via a fairly straightforward adaptation of the main part of the proof of Theorem 2 in~\cite{BFM}; we omit the details. Theorem~\ref{th::FewCyclesExpansion} suggests the following course of action: find a small number of pairwise vertex-disjoint cycles in a given graph, covering  most of its vertices, and then connect them and the remaining vertices so as to obtain one spanning cycle. The following result asserts that such a collection of cycles exists in regular graphs which are not too sparse.

\begin{theorem} \label{th::regGraph2factor}
Let $G$ be an $r$-regular $n$-vertex graph, where $r := r(n) = \omega \left(\sqrt{n \log n} \right)$. Then, $G$ contains pairwise vertex-disjoint cycles $C_1, \ldots, C_k$ such that $k + |V(G) \setminus (V(C_1) \cup \ldots \cup V(C_k))| \leq 2 \sqrt{n^3 \log n}/r$.
\end{theorem}

The proof of Theorem~\ref{th::regGraph2factor} is a straightforward adaptation of the proof of Corollary 2.8 in~\cite{FKS} where the same result is proved for $r$ which is linear in $n$; we omit the details. 

\begin{proof} [Proof of Theorem~\ref{th::DenseRegular}]
If $p \geq (1 + o(1)) \log n/n$, then ${\mathbb G}(n,p)$ is a.a.s. Hamiltonian and thus so is $G \cup {\mathbb G}(n,p)$. Assume then that $p = \frac{c \sqrt{n^3 \log n}}{r^3} < \frac{(1 + o(1)) \log n}{n}$, implying that $r = \Omega \left(\frac{n^{5/6}}{(\log n)^{1/6}} \right)$. Applying Theorem~\ref{th::regGraph2factor} to $G$ implies that it contains pairwise vertex-disjoint cycles $C_1, \ldots, C_k$ such that $k + |V(G) \setminus (V(C_1) \cup \ldots \cup V(C_k))| \leq 2 \sqrt{n^3 \log n}/r$. Since, moreover, $G$ is clearly an $(r/3, 2)$-expander, it follows by Theorem~\ref{th::FewCyclesExpansion} that $G \cup {\mathbb G}(n,p)$ is a.a.s. Hamiltonian. 
\end{proof} 
   
It would be interesting to find additional properties of $G$ (other than having a large minimum degree, a small independence number, or being regular) for which $G \cup \mathbb{G}(n,p)$ is a.a.s. Hamiltonian (or even pancyclic) for relatively small values of $p := p(n)$.

\section*{Acknowledgements} We thank Noga Alon for drawing our attention to reference~\cite{CGGHK}.


\begin{thebibliography}{99}

\bibitem{AF}
M. Anastos and A. Frieze, How many randomly colored edges make a randomly colored dense graph rainbow hamiltonian or rainbow connected?, \emph{Journal of  Graph Theory} 92 (2019), 405--414.

\bibitem{AH}
E. Aigner-Horev and D. Hefetz, Rainbow Hamilton cycles in randomly coloured randomly perturbed dense graphs,
\emph{SIAM Journal on Discrete Mathematics} 35 (2021), 1569--1577.

\bibitem{BBS}
D. Bauer, H. Broersma and E. Schmeichel, Toughness in Graphs --- A Survey, \emph{Graphs and Combinatorics} 22 (2006), 1--35. 

\bibitem{BHKM}
W. Bedenknecht, J. Han, Y. Kohayakawa and G. O. Mota, Powers of tight Hamilton cycles in randomly perturbed hypergraphs, \emph{Random Structures and Algorithms} 55 (2019), 795--807.

\bibitem{BKS}
I. Ben-Eliezer, M. Krivelevich and B. Sudakov, The size Ramsey number of a directed path,
\emph{Journal of Combinatorial Theory Series B} 102 (2012), 743--755.

\bibitem{BFM}
T. Bohman, A. Frieze and R. Martin, How many random edges make a dense graph Hamiltonian?
\emph{Random Structures and Algorithms} 22 (2003), 33--42.

\bibitem{BFKM}
T. Bohman, A. Frieze, M. Krivelevich and R. Martin, Adding random edges to dense graphs,
\emph{Random Structures and Algorithms} 24 (2004), 105--117. 

\bibitem{Bollobas}
B. Bollob\'as, The evolution of sparse graphs, in: \emph{Graph Theory and Combinatorics
(Cambridge, 1983), Academic Press, London}, (1984), 35--57.

\bibitem{BT}
B. Bollob\'as and A. Thomason, Highly linked graphs, \emph{Combinatorica} 16 (1996), 313--320.

\bibitem{Bondy}
 J.A. Bondy, Pancyclic graphs I, \emph{Journal of Combinatorial Theory Ser. B} 11 (1971), 80--84.

\bibitem{BondyConj}
 J.A. Bondy, Pancyclic graphs: Recent results, infinite and finite sets, in: \emph{Colloq. Math. Soc. János Bolyai, Keszthely, Hungary}
(1973), 181--187.

\bibitem{CGGHK}
K. Censor-Hillel, M. Ghaffari, G. Giakkoupis, B. Haeupler and F. Kuhn, Tight Bounds on Vertex Connectivity Under Sampling, \emph{ACM Transactions on Algorithms} 13 (2017), 1--26.

\bibitem{Chvatal}
V. Chv\'atal, Tough graphs and hamiltonian circuits, \emph{Discrete Mathematics} 5 (1973), 215--228. 

\bibitem{CE}
V. Chv\'atal and P. Erd\H{o}s, A note on Hamiltonian circuits, \emph{Discrete Mathematics} 2 (1972), 111--113.

\bibitem{CP}
C. Cooper and A.M. Frieze, Pancyclic random graphs, in: \emph{Proceedings of Random
Graphs ’87}, Edited by M.Karo\'nski, J.Jaworski and A.Ruci\'nski (1990), 29--39.

\bibitem{Dirac}
G. A. Dirac, Some theorems on abstract graphs, \emph{Proc. London Math. Soc.} (3) 2 (1952), 69--81.

\bibitem{DMS}
N. Draganic, D. Munha Correia and B. Sudakov, Erd\H{o}s's conjecture on the pancyclicity of Hamiltonian graphs, preprint.

\bibitem{DRRS}
A. Dudek, C. Reiher, A. Ruci\'nski, and M. Schacht, Powers of hamiltonian cycles in randomly
augmented graphs, \emph{Random Structures and Algorithms} 56 (2020), 122--141.

\bibitem{ER}
P. Erd\H{o}s and A. R\'enyi, On random graphs I, \emph{Publ. Math. Debrecen} 6 (1959), 290--297.

\bibitem{ER2}
P. Erd\H{o}s and A. R\'enyi, On the evolution of random graphs, \emph{Publ. Math. Inst. Hungar. Acad. Sci.} 5 (1960) 17--61.

\bibitem{ER3}
P. Erd\H{o}s and A. R\'enyi, On random matrices, \emph{Publications of the Mathematical Institute of the Hungarian Academy of Sciences} 8 (1964), 455--461.


\bibitem{FKS}
A. Ferber, M. Krivelevich and B. Sudakov, Counting and packing Hamilton cycles in dense graphs and oriented graphs,
\emph{Journal of Combinatorial Theory Series B} 122 (2017), 196--220.  

\bibitem{Frieze}
A. M. Frieze, An algorithm for finding hamilton cycles in random digraphs, \emph{Journal of Algorithms} 9 (1988), 181--204.

\bibitem{FKbook}
A. Frieze and M. Karo\'nski, \textbf{Introduction to Random Graphs}, Cambridge University Press, Cambridge, United Kingdom, 2016.
 
\bibitem{Gould}
R. J. Gould, Recent advances on the Hamiltonian problem: Survey III, \emph{Graphs and Combinatorics} 30 (2014), 1--46. 
 
\bibitem{HZ}
J. Han and Y. Zhao, Embedding Hamilton $\ell$-cycles in randomly perturbed hypergraphs, arXiv:1802.04586, 2018.

\bibitem{KO}
D. K\"uhn and D. Osthus, Hamilton cycles in graphs and hypergraphs: an extremal perspective, \emph{Proceedings of the International Congress of Mathematicians} --- Seoul 2014. Vol. IV, 2014, pp. 381--406. 

\bibitem{KS}
P. Keevash and B. Sudakov, Pancyclicity of Hamiltonian and highly connected graphs, \emph{Journal of Combinatorial Theory Series B} 100 (2010), 456--467.

\bibitem{KSz}
J. Koml\'os and E. Szemer\'edi, Limit distributions for the existence of Hamilton circuits
in a random graph, \emph{Discrete Mathematics} 43 (1983), 55--63.

\bibitem{Korshunov}
A. D. Korshunov, Solution of a problem of Erd\H{o}s and R\'enyi on Hamilton cycles in non-oriented graphs, \emph{Soviet Math. Dokl.} 17 (1976), 760--764.

\bibitem{K}
M. Krivelevich, Long paths and Hamiltonicity in random graphs. In: \emph{Random Graphs, Geometry and Asymptotic Structure, N. Fountoulakis and D. Hefetz, Eds., London Math. Soc. Student Texts Vol. 84, 2016, Cambridge Univ. Press}, pp. 4--27.

\bibitem{KKS}
M. Krivelevich, M. Kwan and B. Sudakov, Cycles and matchings in randomly perturbed digraphs and hypergraphs, \emph{Combinatorics, Probability and Computing} 25 (2016), 909--927.

\bibitem{KLS}
M. Krivelevich, E. Lubetzky and B. Sudakov, Longest cycles in sparse random digraphs, \emph{Random Structures and Algorithms} 43 (2013), 1--15.  

\bibitem{Mader}
W. Mader, Existenz n-fach zusammenh\"angender Teilgraphen in Graphen
gen\"ugend grosser Kantendichte, \emph{Abh. Math. Sem. Univ. Hamburg} 37 (1972),
pp. 86--97.

\bibitem{MM}
A. McDowell and R. Mycroft, Hamilton $\ell$-cycles in randomly perturbed hypergraphs, \emph{Electronic Journal of Combinatorics} 25 (2018), Paper 4.36. 

\bibitem{Posa}
L. P\'osa, Hamiltonian circuits in random graphs, \emph{Discrete Mathematics} 14 (1976), 359--364.

\bibitem{RS}
N. Robertson and P. Seymour, Graph Minors XIII, The disjoint paths problem, \emph{Journal of Combinatorial Theory Series B} 63 (1995), 65--100.

\bibitem{TW}
R. Thomas and P. Wollan, An improved linear edge bound for graph linkages, \emph{European Journal of Combinatorics} 26 (2005), 309--324.

\end{thebibliography}
\end{document}